\documentclass[reqno]{llncs}
\usepackage[margin=1.5in]{geometry}
\usepackage{amsmath}

\usepackage{amsfonts}
\usepackage{amssymb}

\usepackage{lipsum}
\usepackage{ntheorem}

\newcommand{\N}{{\mathbb N}}

\newcommand{\R}{\mathbb{R}}
\newcommand{\Q}{\mathbb{Q}}

\newcommand{\supp}{\text{supp}}

\newcommand{\x}{\textbf{x}}

\newtheorem*{theorem-non}{Theorem}
\newcounter{repc}
\newtheorem*{pr}{Proposition \ref{prop:P3DMalt}}
\newtheorem{cor2}{Corollary}
\newtheorem*{lm2}{Lemma \ref{lm:mintrip}}
\newtheorem*{thm2}{Theorem \ref{thm:baryvalue}}

\usepackage{algorithm}
\usepackage[noend]{algpseudocode}
\usepackage{tikz}

\usepackage[colorlinks=true,breaklinks=true,bookmarks=false,urlcolor=blue,citecolor=blue,linkcolor=blue,bookmarksopen=false,draft=false]{hyperref}

\begin{document}

\title{On the Computational Complexity of \\ Finding a Sparse Wasserstein Barycenter}

\author{Steffen Borgwardt\inst{1} \and Stephan Patterson\inst{2}}

\institute{\email{\href{mailto:steffen.borgwardt@ucdenver.edu}{steffen.borgwardt@ucdenver.edu}};
University of Colorado Denver
\and \email{\href{mailto:stephan.patterson@ucdenver.edu}{stephan.patterson@ucdenver.edu}}; University of Colorado Denver
}

\date{}

\maketitle

\begin{abstract}
The discrete Wasserstein barycenter problem is a minimum-cost mass transport problem for a set of probability measures with finite support. In this paper, we show that finding a barycenter of sparse support is hard, even in dimension $2$ and for only $3$ measures. We prove this claim by showing that a special case of an intimately related decision problem SCMP -- does there exist a measure with a non-mass-splitting transport cost and support size below prescribed bounds? -- is NP-hard for all rational data. Our proof is based on a reduction from planar 3-dimensional matching and follows a strategy laid out by Spieksma and Woeginger (1996) for a reduction to planar, minimum circumference 3-dimensional matching. While we closely mirror the actual steps of their proof, the arguments themselves differ fundamentally due to the complex nature of the discrete barycenter problem. Containment of SCMP in NP will remain open. We prove that, for a given measure, sparsity and cost of an optimal transport to a set of measures can be verified in polynomial time in the size of a bit encoding of the measure. However, the encoding size of a barycenter may be exponential in the encoding size of the underlying measures.


\end{abstract}

\section*{Declarations}

\noindent\textbf{Funding.  }
S. Borgwardt gratefully acknowledges support through the Collaboration Grant for Mathematicians  {\em Polyhedral Theory in Data Analytics} of the Simons Foundation.

\section{Introduction}

Optimal transport problems with multiple marginals arise in a variety of fields and applications. The Wasserstein distance is used as the metric for the transport distance between marginals due to its beneficial properties, including that the resulting solutions, called {\em barycenters}, maintain the geometric structure of the original measures. For two continuous measures $\mu$ and $\nu$ with support in $\R^d$, the Wasserstein distance is given by \[ W_2(\mu,\nu) = \inf \Big \{ \Big(\int_{\R^d \times \R^d} \|x-y\|^2 d\gamma(x,y)\Big)^{\frac{1}{2}}, \gamma \in \Pi(\mu,\nu) \Big \},\]
where $\Pi(\mu,\nu)$ denotes the set of all measures on $\R^d \times \R^d$ with $\mu, \nu$ as marginals. For further details on the Wasserstein distance and its uses, we refer the reader to \cite{pz-19,v-09}. 

The computation of barycenters for continuous probability measures has proven challenging, in part due to the prohibitive cost of calculating the Wasserstein distance. For measures with finite support, the computational complexity of the problem is open, and significant effort continues on efficient exact, approximate, and heuristic solution strategies. In this paper, we address the complexity of finding barycenters with small support sets, called {\em sparse barycenters.} We formally address the distinction between dense and sparse barycenters in Section \ref{subsec:sparse}.

Exact, sparse barycenters can be computed through exponential-sized linear programs \cite{abm-16}. By contrast, most approximation strategies for the discrete barycenter problem, including the active field of entropy-regularized problems \cite{bccnp-14,cd-14,twk-18,ylst-19}, produce approximate measures with {\em dense} support. These methods are popular because they allow for large-scale computations.
However, sparse solutions are desirable for many applications; for instance, the low number of chosen locations is a beneficial trait of a barycenter in facility location problems such as in \cite{abm-16}.  

\subsection{Discrete Barycenters}
Formally, the discrete barycenter problem is defined as follows. Given a finite number $m$ of probability measures $P_1,\ldots, P_m$, each with a finite set of support points $\supp(P_i)$ in $\R^d$, and a corresponding set of positive weights $\lambda_1, \ldots, \lambda_m$ such that $\sum_{i=1}^m \lambda_i = 1$, a Wasserstein barycenter is a probability measure $\bar P$ on $\R^d$ satisfying
\begin{equation*}
\phi(\bar P):=\sum\limits_{i=1}^m \lambda_i W_2(\bar P, P_i)^2 = \inf\limits_{P\in \mathcal{P}^2(\R^d)} \sum\limits_{i=1}^m \lambda_i W_2(P,P_i)^2,
\end{equation*}
where $W_2$ is the quadratic Wasserstein distance and $\mathcal{P}^2(\R^d)$ is the set of all probability measures on $\R^d$ with finite second moments. The function $\phi(P)=\sum_{i=1}^m \lambda_i W_2(P, P_i)^2$ denotes the cost of an {\em optimal transport} of a measure $P$ to measures $P_1,\dots,P_m$. Essentially, a barycenter is a measure $\bar P$ with an associated transport of minimum cost.

When the measures $P_1, \ldots, P_m$ have a finite number of support points, we call them {\em discrete}. When the measures $P_1, \ldots, P_m$ are discrete, the Wasserstein distance is the aggregated squared Euclidean distance, a minimum instead of an infimum can be used, there may exist multiple barycenters, and all barycenters $\bar P$ are discrete \cite{abm-16}. 

We formally state the search for a barycenter of a discrete set of measures as follows, with the additional assumption that the support points in $\supp(P_i), i = 1, \ldots, m$, and their corresponding masses, as well as the weight vector $\lambda$, are rational. The set of rational numbers $\Q$ is an appropriate choice to represent arithmetic computations.
\\\\
\noindent{\bf Discrete Barycenter Problem}
 
\noindent \textbf{Input:} Discrete probability measures $P_1, \ldots, P_m$ in $\Q^d$, weight vector $\lambda\in \Q^m_{+}$

\noindent \textbf{Output:} Discrete barycenter $\bar P$ for $P_1, \ldots, P_m$ and $\lambda$.
\\\\
Computing the transport cost $\phi(\bar P)$ of a barycenter $\bar P$ requires not only the support points and their masses, but also a {\em transport plan}: the support points in each $P_1, \ldots, P_m$ to which each support point in $\bar P$ transports mass, and the amount of mass transported. 

All discrete barycenters satisfy a {\em non-mass-splitting property}: each support point of a barycenter transports its full mass to exactly one point in each input measure $P_i$, $i = 1, \ldots, m$. Thus, the transport plan of a barycenter consists of exactly one point in each measure for each barycenter support point, with associated mass equal to the mass of the barycenter support point. Therefore a barycenter can be described in terms of its transport plan; that is, a barycenter is an assignment of appropriate mass $w_j$ to a subset of the {\em combinations} or {\em tuples} of input support points, denoted $S^* =  \{ (\x_{1}, \ldots, \x_{m}) : \x_{i} \in \supp(P_i)\}$, with elements $s_j = (\x_1^j, \x_2^j, \ldots, \x_m^j)$, $j = 1,\ldots, |S^*|$. 

A consequence of measuring transport costs through the squared Euclidean distance is that the possible locations of barycenter support points are given by the optimal locations for placing mass for each of the tuples. That is, when mass is associated to a tuple $s_j$, the barycenter support point for the mass is given by the {\em weighted mean} $\x^j = \sum_{i=1}^m \lambda_i \x_i^j$. The {\em set $S$ of unique weighted means} contains all possible support points for a barycenter. Note that since the support points of $P_1, \ldots, P_m$ are rational by assumption, so too are the elements of $S$ and the support points of $\bar P$. 

Any measure $P$ with a non-mass-splitting transport plan to $P_1, \ldots, P_m$, can be described as a set of tuples from $S^*$ and corresponding masses; however, the transport cost associated with a tuple depends on the chosen location for mass assignment. Since the optimal location for placing mass for minimizing the transport to a given tuple is the weighted mean, we define a {\em combination measure} $P^*$ to be a measure that is provided as a set of tuples and whose support points are weighted means that can be calculated from the tuples. 

We now have two ways to represent a barycenter. A measure denoted as $P$ is given in the traditional manner, namely, a set of support points from $S$ and corresponding masses. A combination measure denoted as $P^*$ is given as a set of tuples from $S^*$ and corresponding masses. Since all barycenters have non-mass-splitting optimal transport plans, the set $\mathcal{C}$ of all combination measures includes all barycenters, and so we may restrict the search for a barycenter to $\mathcal{C}$. Note this is not a restriction on $S$, the possible support points for $\bar P$, but rather a restriction on the possible transport plans. 

A measure $P^*$ provides a natural, implicit transport plan, while a measure in the form $P$ does not.  In Section \ref{sec:open}, we examine the transformation between these two representations for barycenters and combination measures. In this paper, we primarily work with barycenters provided in the form $P^*$.

 

\subsection{Sparse Barycenters}\label{subsec:sparse}
We consider a variant of the discrete barycenter problem with a requirement on the size of the set $\supp(\bar P)$. Letting $|P_i|$ denote the size of $\supp{(P_i)}$, it is known that there always exists a barycenter $\bar P$ with a support set of size at most \begin{equation*} |\bar P| \leq \sum_{i=1}^m |P_i| -m+1. \end{equation*} Furthermore, a barycenter with such sparse support can be computed by finding an optimal vertex for a linear program \cite{abm-16}. As mentioned previously, sparse support is beneficial in many applications, and indeed this bound is very small compared to the number of possible tuples, $|S^*|$, which is exponential in the number of input measures $m$.

We formally define a variant of the discrete barycenter problem in which we prescribe an upper bound $N$ on the number of support points `allowed' for a barycenter as follows. We always assume $N \leq \sum_{i=1}^m |P_i| -m+1$, so that any measure which satisfies the requirements of SBP is at least as sparse as one guaranteed to exist. 
\\\\
\noindent{\bf Sparse Barycenter Problem [SBP]}
 
\noindent \textbf{Input:} Discrete probability measures $P_1, \ldots, P_m$ in $\Q^d$, weight vector $\lambda\in \Q^m_{+}$, support size bound $N\in \N$

\noindent \textbf{Output:} Discrete barycenter $\bar P$ for $P_1, \ldots, P_m$ and $\lambda$, with a support set $\supp(\bar P)$ of size at most $N$ (or decision that none exists)
\\\\
We assume that the size of a bit encoding of SBP (and the other problems throughout this paper) is dominated by the input of the measures $P_1, \ldots, P_m$, i.e., a bit encoding of masses and the coordinates of support points. Note that an encoding has to contain at least one bit for each support point in each measure and thus its size is larger than the $N$ we are interested in.

A solution to SBP can be found through an examination of the underlying polyhedron, given by the mass transport requirements of $P_1, \ldots, P_m$ \cite{abm-16}. An inclusion-minimal support set occurs at a vertex of the polyhedron. Thus in principle, SBP could be solved through an enumeration of the vertices of the optimal face, checking for sufficiently small support. However, the associated linear program may be exponentially-sized, and the number of vertices of the face may be exponential, so one obtains no immediate information about the hardness of SBP from these observations. 
 
 \subsection{Outline of the Paper}
 
In Section \ref{sec:proofstrat}, we state our main results regarding the difficulty of finding barycenters with small support sets. We begin in Section \ref{subsec:results} with a definition of a decision problem SCMP on the existence of a combination measure with bounds on the size of the support set and the transport cost. We show this problem is NP-hard through a reduction from a well-known NP-hard problem; the decision problem and the strategy for the reduction are outlined in Section \ref{subsec:proofstrat}. This result also implies that an efficient algorithm for SBP cannot exist, unless P = NP. Our reduction depends heavily on a particular type of planar graph, which we describe in Section \ref{sec:G2}. Section \ref{sec:HardProof} contains the details of the proofs for the reduction, as stated in Section \ref{sec:proofstrat}.

In Section \ref{sec:open}, we discuss why our approach does not readily provide an answer to containment in NP, respectively the efficient verification of some of the properties of a given measure. In Section \ref{subsec:SCMPNP}, we turn to SCMP; in Section \ref{subsec:SBPver}, we turn to SBP. 
Finally, we conclude with a brief outline of further open questions on the complexity of the discrete barycenter problem in Section \ref{sec:openQ}. 
 

\section{Results and the Strategy of the Proof}\label{sec:proofstrat}

Our result that SBP cannot be solved efficiently (unless $P = NP$) follows from a reduction of a well-known NP-hard problem to a decision variant of SBP. We begin with two problem definitions: first, a decision problem that is more general than the sparse barycenter problem; then, a special case of this problem used in our reduction. Both are shown to be NP-hard in this paper.


\subsection{Results}\label{subsec:results}

For a decision variant of SBP, we replace the output of SBP -- a sparse barycenter -- with a question regarding the existence of a sparse measure with a non-mass-splitting transport of cost below a prescribed bound. Recall that barycenters satisfy the non-mass-splitting property and therefore can be given as a combination measure. Therefore, the search space for a barycenter may be restricted to the set $\mathcal{C}$ of measures for which a non-mass-splitting transport exists and whose support points are the weighted mean of the destination points. 
\\\\\\
\noindent{\bf Sparse Combination Measure Problem [SCMP]}
 
\noindent \textbf{Input:} Discrete probability measures $P_1, \ldots, P_m$ in $\Q^d$, weight vector $\lambda\in \Q^m_{+}$, support size bound $N\in \N$, transport cost bound $\Phi$

\noindent \textbf{Decide:} Does there exist a combination measure $P^*$ with a support set $\supp(P^*)$ of size at most $N$ and transport cost $\phi(P^*)\leq \Phi$?
\\\\
A measure $P^*$ which satisfies the requirements of SCMP is an approximation of a barycenter with the beneficial non-mass-splitting and sparsity properties of an exact barycenter. 

We show that SCMP is NP-hard for $m \geq 3$ and $d \geq 2$. Our proof that SCMP is NP-hard follows from a reduction to a heavily restricted set of SCMP instances, in which
\begin{itemize}
\item there are exactly three measures $P_1,P_2,P_3$ in the Euclidean plane, i.e., $m=3$ and $d=2$,
\item all three measures have the same number of support points, $|P_1| = |P_2| = |P_3| = n$,
\item every measure has evenly distributed mass, i.e., all support points have mass $1/n$,
\item and the weights of the measures $\lambda_i$ are equal, so that $\lambda_i = 1/3$ for $i = 1, 2, 3$.
\end{itemize}

We call SCMP, restricted to such instances, the Uniform Combination 3-Measure Problem. \\

\noindent{\bf Uniform Combination 3-Measure Problem [UC3P]}

\noindent \textbf{Input:} Discrete probability measures $P_1, P_2,$ and $P_3$, equally weighted with $\lambda_i = 1/3$, each with $n$ support points in the Euclidean plane and uniformly distributed mass, support size bound $N\in \N$, transport cost bound $\Phi$

\noindent \textbf{Decide:} Does there exist a combination measure $P^*$ with a support set  $\supp(P^*)$ of size at most $N$ and transport cost $\phi(P)\leq \Phi$?
\\\\
Despite the various restrictions to the input, UC3P is still NP-hard.
\begin{theorem}\label{thm:UC3P}
UC3P is NP-hard.
\end{theorem}

In Section \ref{subsec:proofstrat}, we explain the strategy to prove Theorem \ref{thm:UC3P} -- the actual work is done in Sections \ref{sec:G2} and \ref{sec:HardProof}. NP-hardness of UC3P immediately implies NP-hardness of SCMP for $m=3$ and $d=2$. A generalization to any $m\geq 3$ and $d\geq 2$ is trivial -- one can include additional measures that have a single support point, respectively lift the coordinates to an affine subspace of a higher-dimensional space. 

\begin{theorem}\label{thm:SCMPNP}
SCMP is NP-hard for $m \geq 3$ and $d \geq 2$.
\end{theorem}

An argument that an efficient algorithm for SBP must not exist follows from the NP-hardness of UC3P. Suppose an efficient algorithm for SBP exists for all $m$ and sparsities $N$. Then when $m=3$, a decision for UC3P with sparsity $N$ and $\Phi$ equal to the transport value of a barycenter would be efficient. As is shown in Section \ref{sec:HardProof}, this would lead to an efficient decision of a well-known NP-hard problem, a contradiction.

\begin{theorem}\label{thm:SBPNP}
For $m\geq 3$ and $d \geq 2$, an efficient algorithm for SBP cannot exist, unless P = NP. 
\end{theorem}



Finally, for $m=2$ or $d = 1$, the hardness of SCMP will remain open. For $m=2$, the search for a discrete barycenter takes the form of a classical transportation problem (of polynomial size in the input) \cite{bp-18,m-16}, which is efficiently solvable. Finding a vertex of the underlying polytope leads to a barycenter $\bar P$ with $|\bar P| \leq \sum_{i=1}^m |P_i| -m+1$. When $d =1$, that is, when the measures $P_1,\ldots,P_m$ have support points in $\R$, their barycenter is unique and can be found through ordering the support points of each input measure $P_i$, $i = 1, \ldots, m$, either from least to greatest or vice versa, and then constructing the support points of the barycenter greedily in the same fashion. For example, in \cite{rpdb-12} it is shown that barycenters in $\R$ can be computed in $\mathcal{O}(m \log m)$, which forms the basis for an approximation scheme producing the so-called Sliced Wasserstein barycenters. However, the efficient algorithms for barycenters in these cases do not translate to an efficient decision of SCMP for an arbitrary $N$.
\vspace{-10pt}
\subsection{Strategy of the Proof}\label{subsec:proofstrat}
We prove UC3P to be NP-hard through a (polynomial) reduction of planar three-dimensional matching to UC3P. Planar three-dimensional matching is formally defined as follows.
\\
\noindent{\bf Planar 3-Dimensional Matching [P3DM]}

\noindent\textbf{Input:} Three pairwise disjoint sets $X,Y,$ and $Z$, and a set $T \subseteq X \times Y \times Z$ such that
	\begin{enumerate}
	\item $|X|=|Y|=|Z| = q$
	\item Every element of $X \cup Y \cup Z$ occurs in at most three triples from $T$
	\item The induced graph $G_{\text{P3DM}}$, containing a vertex for each element of $X \cup Y \cup Z$ and for every triple in $T$, and an edge connecting each triple to the elements contained within, is planar.
	\end{enumerate}

\noindent\textbf{Decide:} Does there exist a subset $T' \subseteq T$ such that each element of $X \cup Y \cup Z$ is contained in exactly one triple from $T'$?
\\\\
P3DM is a special case of three-dimensional matching, one of the most famous problems in combinatorial optimization known to be NP-hard \cite{k-72,sw-96}. P3DM itself is NP-hard \cite{df-86}. Thus, a polynomial reduction to UC3P will prove Theorem \ref{thm:UC3P}. To this end, we (efficiently) construct a specific type of instances for which the ability to efficiently decide UC3P would imply the ability to efficiently decide planar three-dimensional matching for any input. 

The construction of these instances builds on the proof strategy of Spieksma and Woeginger in \cite{sw-96} on the hardness of a special matching problem, defined as follows.
\\\\
\noindent{\bf Minimum Circumference Matching Problem [MCM]}


\noindent\textbf{Input:} Three sets $B$ (blue), $R$ (red), and $G$ (green) containing points with integer coordinates in the Euclidean plane.

\noindent\textbf{Output:} A partition of the set $B \cup R \cup G$ containing $n$ three-colored triples $t_j = (b^j,r^j,g^j)$, such that the total cost $\sum_{j=1}^n c(t_j)$ is minimal, where $c(t_j) = ||b^j-r^j|| + ||b^j-g^j|| + ||r^j-g^j||$. 
\\\\
Note that for any instance of P3DM, there is an associated induced graph $G_{\text{P3DM}}$ (property $3$). In order to reduce P3DM to UC3P, we construct a second planar graph $G_{\text{TP}}$ from $G_{\text{P3DM}}$, similar to that constructed in \cite{sw-96} but containing a few additional assumptions; the details of the construction are in Section \ref{sec:G2}. The graph $G_{\text{TP}}$ is built using so-called {\em triangle paths}, defined in Section \ref{subsec:graph}. There are two reasons why $G_{\text{TP}}$ is important: first, to reduce P3DM to UC3P, we construct a simple instance of a barycenter problem that mirrors $G_{\text{TP}}$ almost exactly. Second, just like in the proof of NP-hardness of MCM, our goal is to decide whether there exists a {\em pattern of alternating triangles} in this graph, defined formally in Section \ref{subsec:alt}. 

\begin{proposition}\label{prop:P3DMalt}
Given any instance of P3DM, construct a graph $G_{\text{TP}}$. There exists a pattern of alternating triangles in $G_{\text{TP}}$ if and only if P3DM has answer YES.
\end{proposition}

While our general strategy closely follows the proof in \cite{sw-96}, the (still) continuous nature of the discrete barycenter problem, respectively UC3P -- combination measures could have any number of support points of any mass between $0$ and $1$ --  makes fundamental differences in the proofs arise almost immediately. As we will see, the details and potential pitfalls of the analysis are more involved.

Theorem \ref{thm:SCMPNP} follows from a series of smaller results using $G_{\text{TP}}$ to construct input for UC3P. For any instance of P3DM and an associated graph $G_{\text{TP}}$, $\mathcal{U}$ will denote an instance of UC3P with sparsity $N$ and transport bound $\Phi$ whose input measures $P_1,\ldots,P_m$ are constructed from $G_{\text{TP}}$. If P3DM is assumed to be a YES-instance, we will denote the corresponding instance $\mathcal{U}$ as $\mathcal{U_Y}$. Because UC3P has exactly three input measures $P_1, P_2, P_3$, we call a tuple from $S^*$ a {\em triple} and replace the notation $s_j$ with $t_j$.

In the following, we find the minimum value $\Phi$, then the minimum sparsity $N$, which can appear in a YES-instance of $\mathcal{U}$. A proof for each statement is given in Section \ref{sec:HardProof}. 
First, we observe that for all instances of P3DM, the minimum transport cost of a unit of mass among all triples in $G_{\text{TP}}$ is $50/9$.

\begin{lemma}[Minimum Cost Triple]\label{lm:mintrip}
For any instance $\mathcal{U}$, the minimum cost among all possible triples $t_j$ is $c(t_j) = 50/9$.
\end{lemma}

For a YES-instance of P3DM, Lemma \ref{lm:mintrip} allows the determination of the cost $\phi(\bar P)$ for $\mathcal{U_Y}$, given in Theorem \ref{thm:baryvalue}. This holds due to the structure of $G_{\text{TP}}$: when $G_{\text{TP}}$ is constructed from a YES-instance of P3DM, it is possible to construct a measure containing only triples of minimum cost. 

\begin{theorem}[Transport Cost of a Barycenter]\label{thm:baryvalue} For an instance $\mathcal{U_Y}$, a barycenter has transport cost $\phi(\bar P) = 50/9$. \end{theorem}

Therefore an instance $\mathcal{U}$ has a possible YES answer only if $\Phi \geq 50/9$. It remains to address the existence of a barycenter with minimum possible sparsity. The non-mass-splitting property gives an immediate lower bound on the number of support points in any barycenter: $\max_{i \leq m} |P_i|$. Thus, a barycenter $\bar P$, and by extension any combination measure, with the smallest support set will always have support size
$$\max_{i \leq m} |P_i| \leq |\bar P| \leq \sum_{i=1}^m |P_i| -m+1.$$ This implies that  $N \geq n$ in a YES-instance of UC3P.

Then, the proof of Theorem \ref{thm:baryvalue} demonstrates the existence of a combination measure $P^*$ containing only triples of minimum cost; therefore no lower transport cost is possible and $P^*$ is a barycenter. This measure has $n$ support points by construction; therefore, an immediate consequence of Theorem \ref{thm:baryvalue} is that for an instance $\mathcal{U}$, there exists a barycenter with exactly $n$ support points. 

\begin{corollary}[Existence of a Sparse Barycenter]\label{cor:barysparse} For an instance $\mathcal{U_Y}$, there exists a barycenter with $n$ support points. \end{corollary}

Through Corollary \ref{cor:barysparse} and Theorem \ref{thm:baryvalue}, a YES-instance of P3DM implies a YES-instance of UC3P in the corresponding instance $\mathcal{U}$ with $N \geq n$ and $\Phi \geq 50/9$. In order to show a YES-instance of UC3P in $\mathcal{U}$ implies the corresponding instance of P3DM also has answer YES, we assume there exists a combination measure with $n$ support points and transport cost $50/9$. The combinations from the combination measure create a pattern of alternating triangles in $G_{\text{TP}}$. 

\begin{theorem}[Barycenters Create a Pattern of Alternating Triangles]\label{thm:sparsealt} For an instance $\mathcal{U}$ that has a barycenter $P$ of $n$ support points and transport cost $50/9$, $P$ creates a pattern of alternating triangles. 
\end{theorem}

By Theorem \ref{thm:sparsealt} and Proposition \ref{prop:P3DMalt}, the existence of a barycenter with minimum possible support size $n$ and transport cost $50/9$ implies that the instance of P3DM from which $G_{\text{TP}}$ was constructed must have answer YES. 

\begin{theorem}[P3DM Reduction to UC3P]\label{thm:yesiff}
For an instance $\mathcal{U}$ with $N = n$ and $\Phi = 50/9$, UC3P has answer YES if and only if P3DM has answer YES.
\end{theorem}

While we prove Theorem \ref{thm:sparsealt} using the minimum values $N = n$ and $\Phi =50/9$, an inspection of $G_{\text{TP}}$ reveals that a pattern of alternating triangles is produced in $G_{\text{TP}}$ for larger values of $N$ and $\Phi$ as well. These values can be increased up to thresholds which depend on the exact input of $P_1$, $P_2$, and $P_3$. 

Regardless, the existence of a pattern of alternating triangles in $G_{\text{TP}}$ produced by solutions to UC3P completes the reduction of P3DM to UC3P.

\vspace{15pt}

\begin{thm2}
UC3P is NP-hard. 
\end{thm2}

\begin{proof}
For any instance $\mathcal{U}$, the graph $G_{\text{TP}}$ can be constructed in polynomial time; see Section \ref{subsec:graph}. By Theorem \ref{thm:yesiff}, P3DM and the corresponding UC3P in $\mathcal{U}$ have the same decision. Therefore, if UC3P were efficiently solvable, P3DM would also be efficiently solvable. Since P3DM is NP-hard, UC3P is NP-hard.
 \end{proof} 

The NP-hardness of UC3P connects to both SBP and SCMP. For Theorem \ref{thm:SCMPNP}, if SCMP were efficiently solvable, UC3P would be efficiently solvable. For Theorem \ref{thm:SBPNP}, an efficient algorithm for SBP implies the ability to solve UC3P by simply computing a sufficiently sparse barycenter. 


\section{A Review of the P3DM Graph $G_{\text{TP}}$}\label{sec:G2}

The graph $G_{\text{TP}}$ is used in an instance $\mathcal{U}$ to connect an instance of P3DM to a problem UC3P. In this section, we begin with the construction of the graph. We follow the process described in \cite{sw-96} (which itself is based on a similar construction in \cite{prw-94}) with a few minor changes. Then we turn to a pattern of alternating triangles as mentioned in Proposition \ref{prop:P3DMalt}. This pattern appears in $G_{\text{TP}}$ only for YES-instances of the instances considered in \cite{sw-96}.  We show in Section \ref{sec:HardProof} that this pattern is also produced by solutions to UC3P when $G_{\text{TP}}$ is used to form the input probability measures. We conclude this section with a description of the transformation of $G_{\text{TP}}$ to probability measures for an instance for UC3P. 

\subsection{Construction of the Graph $G_{\text{TP}}$}\label{subsec:graph}

The construction of $G_{\text{TP}}$ begins with an instance of P3DM and its associated induced planar graph $G_{\text{P3DM}}$. First, a rectilinear layout of $G_{\text{P3DM}}$ is computed, which can be done efficiently. This places all vertices at integer coordinates, and ensures that all edges are either vertical or horizontal. Next, the coordinates of all points are shifted, multiplying by 1000, so that there is significant vertical and horizontal space between the vertices.

The graph $G_{\text{TP}}$ is now built from repetition of a single basic `building block': a right triangle with integer side lengths of $3$, $4$, and $5$, hereafter called a {\em 3-4-5 triangle}. The legs of all 3-4-5 triangles are aligned vertically and horizontally in the plane. Each vertex of $G_{\text{TP}}$ is assigned to one of three sets $P_1, P_2,$ and $P_3$. Each vertex of a 3-4-5 triangle belongs to a different set.

The graph $G_{\text{P3DM}}$ contains two types of vertices: those corresponding to triples in $T$, and those corresponding to elements of $X$, $Y$, and $Z$. In the computed rectilinear layout, for each vertex of $G_{\text{P3DM}}$ that corresponds to a triple in $T$, we replace the vertex with a 3-4-5 triangle. In this case, we call the 3-4-5 triangle a {\em triple triangle}. Each vertex of a triple triangle represents the element from $X$, $Y$, or $Z$ in the triple from $T$. 

The remaining vertices of $G_{\text{P3DM}}$, those corresponding to elements of $X$, $Y$, and $Z$, remain unchanged in $G_{\text{TP}}$ and are called {\em element points}. One edge connects each element point to the vertex of the triple triangle representing that element. Since an element of $X$, $Y$, and $Z$ appears in at most three triples in $T$, at most three edges leave an element point.

Each of these edges is now replaced with a so-called {\em triangle path}, constructed out of rectangular pairs of 3-4-5 triangles, starting from the element point and ending at a vertex of the triple triangle. The vertices of each triangle in the triangle path are vertices in $G_{\text{TP}}$. The paths are constructed in such a way as to meet the following requirements: 
\begin{enumerate}
\item The first pair of 3-4-5 triangles in each path is oriented vertically, that is, with the longer edge (4) vertical. 
\item The second pair of 3-4-5 triangles is connected to the first at just one vertex, and is oriented horizontally, that is, with the longer edge horizontal. 
\item The orientation of all triangles after the first four are chosen in such a way that the vertices of $G_{\text{TP}}$, and the vertices of every 3-4-5 triangle within, are assigned to three disjoint sets. In particular, the triple triangles, at which three paths end, must also have one vertex in each of the three sets. \end{enumerate}

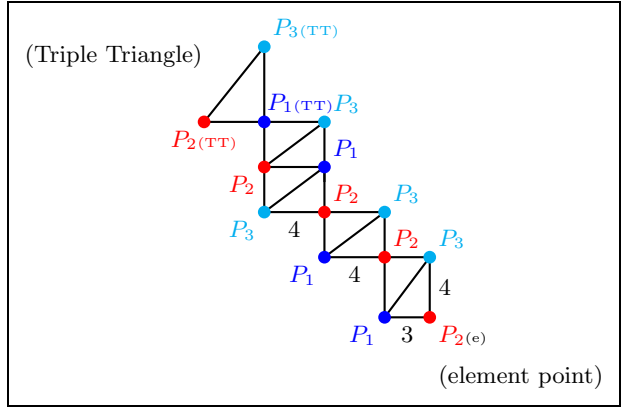
\begin{figure}[!t]
\begin{center}
\fbox{\begin{tikzpicture}[scale=0.4]

\draw [thick](1,2)--(-0.5,2) node [midway, below]{$3$};
\draw [thick] (1,2)--(1,4) node [midway, right]{$4$};
\draw [thick] (1,4)--(-0.5,2);

\draw [thick](-0.5,4)--(-2.5,4) node [midway, below]{$4$};
\draw [thick] (-0.5,4)--(-0.5,5.5);
\draw [thick](-0.5,4)--(1,4);
\draw [thick] (-0.5,4)--(-0.5,2);
\draw [thick] (-0.5,5.5)--(-2.5,4);

\draw [thick](-2.5,5.5)--(-4.5,5.5) node [midway, below]{$4$};
\draw [thick] (-2.5,5.5)--(-2.5,7);
\draw [thick](-2.5,5.5)--(-0.5,5.5);
\draw [thick] (-2.5,5.5)--(-2.5,4);
\draw [thick] (-4.5,5.5)--(-2.5,7);

\draw [thick] (-4.5,7)--(-4.5,5.5);
\draw [thick] (-4.5,7)--(-2.5,7);
\draw [thick] (-4.5,7)--(-2.5,8.5);
\draw [thick] (-4.5,7)--(-4.5,8.5);
\draw [thick] (-4.5,8.5)--(-2.5,8.5);
\draw [thick] (-2.5,6.5)--(-2.5,8.5);

\draw [thick] (-4.5,11)--(-6.5,8.5);
\draw [thick] (-4.5,11)--(-4.5,8.5);
\draw [thick] (-6.5,8.5)--(-4.5,8.5);

\draw (-9.5,10) node [anchor= south]{(Triple Triangle)};
\draw (1,0) node [anchor = west]{(element point)};

\fill [red] (-0.5,4) circle (6pt) node [anchor = south west]{$P_2$};
\fill [red] (1,2) circle (6pt) node [anchor = north west]{$P_2${\color{black}\tiny(e)}};
\fill [blue] (-0.5,2) circle (6pt) node [anchor = north east]{$P_1$};
\fill [cyan] (1, 4) circle (6pt) node [anchor = south west]{$P_3$};
\fill [blue] (-2.5,4) circle (6pt) node [anchor = north east]{$P_1$};
\fill [cyan] (-0.5, 5.5) circle (6pt) node [anchor = south west]{$P_3$};
\fill [red] (-2.5,5.5) circle (6pt) node [anchor = south west]{$P_2$};
\fill [cyan] (-4.5,5.5) circle (6pt) node [anchor = north east]{$P_3$};
\fill [blue] (-2.5, 7) circle (6pt) node [anchor = south west]{$P_1$};
\fill [red] (-6.5,8.5) circle (6pt) node [anchor = north]{ $P_2${\tiny(TT)}};
\fill [cyan] (-4.5,11) circle (6pt) node [anchor = south west]{$P_3${\tiny(TT)}};
\fill [red] (-4.5,7) circle (6pt) node [anchor = north east]{$P_2$};
\fill [cyan] (-2.5,8.5) circle (6pt) node [anchor = south west]{$P_3$};
\fill [blue] (-4.5, 8.5) circle (6pt) node [anchor = south west]{\hspace{-.03in}$P_1${\tiny(TT)}};
\end{tikzpicture}}
\end{center}
\caption{An example of a triangle path. Each element point is connected to the vertex of all triple triangles containing the element by a triangle path. The orientation of the first two pairs of triangles is fixed; thereafter, pairs can be oriented horizontally or vertically, and can be connected at one or two vertices to the previous block. }\label{fig:pathpattern}
\end{figure}

Given sufficient space -- which is provided by the initial shifting of coordinates -- it is always possible to construct triangle paths in this manner \cite{sw-96}. 
We show an example of a triangle path in Figure \ref{fig:pathpattern}. The triangle path begins at an element point, denoted $(e)$ in all figures. The orientation of the first two rectangle pairs is fixed. After the first two pairs, the pairs can be oriented horizontally or vertically, and can be connected at one or two vertices to the previous pair. There are many options for how these triangle paths may be built, roughly following along the corresponding vertical or horizontal edge in the rectilinear layout. However, the path must be chosen in such a way that the vertices in the triple triangle at the end of the paths are assigned to three different sets, $P_1$, $P_2$, $P_3$. 

Recall that at most three paths leave an element point. This is displayed  in Figure \ref{fig:threepath}. When a vertex that is not an element point belongs to two paths, it is listed in its respective set twice. In Figure \ref{fig:threepath}, there are two such points, labeled $(2)$. These points are a minor complication for using $G_{\text{TP}}$ as input for UC3P; the details are addressed in Section \ref{subsec:G2toBP}.

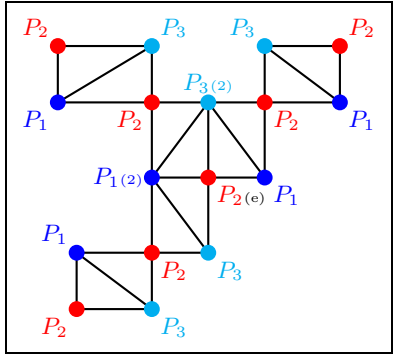
\begin{figure}[t]
\begin{center}
\fbox{\begin{tikzpicture}[scale=0.5]


\draw [thick](1,2)--(-0.5,2);
\draw [thick] (1,2)--(1,4);
\draw [thick] (1,4)--(-0.5,2);

\draw [thick](-0.5,4)--(-3,4);
\draw [thick] (-0.5,4)--(-0.5,5.5);
\draw [thick](-0.5,4)--(1,4);
\draw [thick] (-0.5,4)--(-0.5,2);
\draw [thick] (-0.5,5.5)--(-3,4);

\draw [thick](-3,5.5)--(-0.5,5.5);
\draw [thick] (-3,5.5)--(-3,4);


\draw [thick](1,2)--(2.5,2);
\draw [thick] (2.5,2)--(2.5,4);
\draw [thick] (1,4)--(2.5,2);
\draw[thick] (1,4)--(2.5,4);

\draw[thick](4.5,4)--(2.5,4);
\draw[thick](2.5,4)--(2.5,5.5);
\draw[thick](4.5,4)--(2.5,5.5);
\draw[thick](4.5,4)--(4.5,5.5);
\draw[thick](2.5,5.5)--(4.5,5.5);


\draw[thick](1,2)--(1,0);
\draw[thick](1,0)--(-0.5,2);
\draw [thick] (1,0)--(-0.5,0);
\draw [thick] (-0.5,0)--(-0.5,2);

\draw[thick] (-2.5,-1.5)--(-0.5,-1.5);
\draw[thick] (-2.5,-1.5)--(-2.5,0);
\draw[thick] (-0.5,-1.5)--(-2.5,0);
\draw[thick] (-2.5,0)--(-0.5,0);
\draw[thick] (-0.5,-1.5)--(-0.5,0);

\fill [red] (1,2) circle (6pt) node [anchor = north west]{$P_2${\color{black}\tiny(e)}};
\fill [blue] (-0.5,2) circle (6pt) node [anchor = east]{$P_1${\tiny(2)}};
\fill [cyan] (1, 4) circle (6pt) node [anchor = south ]{$P_3${\tiny (2)}};
\fill [red] (-0.5,4) circle (6pt) node [anchor = north east]{$P_2$};
\fill [blue] (-3,4) circle (6pt) node [anchor = north east]{$P_1$};
\fill [cyan] (-0.5, 5.5) circle (6pt) node [anchor = south west]{$P_3$};
\fill [red] (-3,5.5) circle (6pt) node [anchor = south east]{$P_2$};
\fill [red] (2.5,4) circle (6pt) node [anchor = north west]{$P_2$};
\fill [blue] (2.5,2) circle (6pt) node [anchor = north west]{$P_1$};
\fill [blue] (4.5, 4) circle (6pt) node [anchor = north west]{$P_1$};
\fill [cyan] (2.5, 5.5) circle (6pt) node [anchor = south east]{$P_3$};
\fill [red] (4.5,5.5) circle (6pt) node [anchor = south west]{$P_2$};
\fill [cyan] (1,0) circle (6pt) node [anchor = north west]{$P_3$};
\fill[red] (-0.5,0) circle (6pt) node [anchor = north west]{$P_2$};
\fill[blue] (-2.5,0) circle (6pt) node [anchor = south east]{$P_1$};
\fill[cyan] (-0.5,-1.5) circle (6pt) node [anchor = north west]{$P_3$};
\fill[red] (-2.5,-1.5) circle (6pt) node [anchor = north east]{$P_2$};
\end{tikzpicture}}
\end{center}
\caption{Each element point, shown here as an element of $P_2$, is incident to at most three triangle paths.}\label{fig:threepath}
\end{figure}

The construction of $G_{\text{TP}}$ is now complete. It contains three types of vertices: element points, representing elements of $X$, $Y$, and $Z$; the three vertices of each triple triangle, representing the elements of $X$, $Y$ and $Z$ in that triple of $T$; and vertices of 3-4-5 triangles in the triangle paths, which have no meaning in P3DM. Then, $P_1$, $P_2$, and $P_3$ all contain $n$ points where $n$ is the total number of 3-4-5 triangles in $G_{\text{TP}}$.

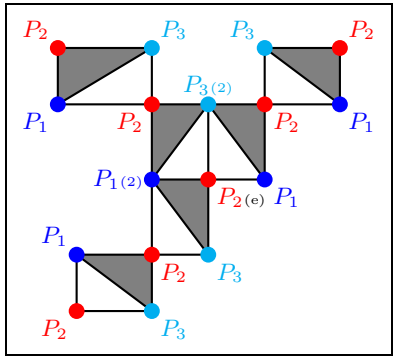
\begin{figure}[t]
\begin{center}
\fbox{\begin{tikzpicture}[scale=0.5]

\path [fill=gray] (1,4)--(-0.5,4)--(-0.5,2);
\path [fill=gray] (-3,4)--(-3,5.5)--(-0.5,5.5);
\path [fill=gray] (1,2)--(1,0)--(-0.5,2);
\path [fill=gray] (-2.5,0)--(-0.5,-1.5)--(-0.5,0);
\path [fill=gray] (2.5,2)--(1,4)--(2.5,4);
\path [fill=gray] (2.5,5.5)--(4.5,5.5)--(4.5,4);

\draw [thick](1,2)--(-0.5,2);
\draw [thick] (1,2)--(1,4);
\draw [thick] (1,4)--(-0.5,2);
\draw [thick](-0.5,4)--(-3,4);
\draw [thick] (-0.5,4)--(-0.5,5.5);
\draw [thick](-0.5,4)--(1,4);
\draw [thick] (-0.5,4)--(-0.5,2);
\draw [thick] (-0.5,5.5)--(-3,4);

\draw [thick](-3,5.5)--(-0.5,5.5);
\draw [thick] (-3,5.5)--(-3,4);

\draw [thick](1,2)--(2.5,2);
\draw [thick] (2.5,2)--(2.5,4);
\draw [thick] (1,4)--(2.5,2);
\draw[thick] (1,4)--(2.5,4);

\draw[thick](4.5,4)--(2.5,4);
\draw[thick](2.5,4)--(2.5,5.5);
\draw[thick](4.5,4)--(2.5,5.5);
\draw[thick](4.5,4)--(4.5,5.5);
\draw[thick](2.5,5.5)--(4.5,5.5);

\draw[thick](1,2)--(1,0);
\draw[thick](1,0)--(-0.5,2);
\draw [thick] (1,0)--(-0.5,0);
\draw [thick] (-0.5,0)--(-0.5,2);

\draw[thick] (-2.5,-1.5)--(-0.5,-1.5);
\draw[thick] (-2.5,-1.5)--(-2.5,0);
\draw[thick] (-0.5,-1.5)--(-2.5,0);
\draw[thick] (-2.5,0)--(-0.5,0);
\draw[thick] (-0.5,-1.5)--(-0.5,0);

\fill[blue] (-2.5,0) circle (6pt) node [anchor = south east]{$P_1$};
\fill[cyan] (-0.5,-1.5) circle (6pt) node [anchor = north west]{$P_3$};
\fill[red] (-2.5,-1.5) circle (6pt) node [anchor = north east]{$P_2$};
\fill [cyan] (1,0) circle (6pt) node [anchor = north west]{$P_3$};
\fill[red] (-0.5,0) circle (6pt) node [anchor = north west]{$P_2$};
\fill [red] (-3,5.5) circle (6pt) node [anchor = south east]{$P_2$};
\fill [red] (-0.5,4) circle (6pt) node [anchor = north east]{$P_2$};
\fill [blue] (-3,4) circle (6pt) node [anchor = north east]{$P_1$};
\fill [cyan] (-0.5, 5.5) circle (6pt) node [anchor = south west]{$P_3$};
\fill [red] (1,2) circle (6pt) node [anchor = north west]{$P_2${\color{black}\tiny(e)}};
\fill [blue] (-0.5,2) circle (6pt) node [anchor = east]{$P_1$\tiny{(2)}};
\fill [cyan] (1, 4) circle (6pt) node [anchor = south]{$P_3$\tiny{(2)}};
\fill [red] (2.5,4) circle (6pt) node [anchor = north west]{$P_2$};
\fill [blue] (2.5,2) circle (6pt) node [anchor = north west]{$P_1$};
\fill [blue] (4.5,4) circle (6pt) node [anchor = north west]{$P_1$};
\fill [cyan] (2.5, 5.5) circle (6pt) node [anchor = south east]{$P_3$};
\fill [red] (4.5,5.5) circle (6pt) node [anchor = south west]{$P_2$};

\end{tikzpicture}}
\end{center}
\caption{Matching the elements of sets $P_1$, $P_2$, and $P_3$ into a partition of $n$ minimum circumference triangles creates alternating triangle paths. The triple triangle is selected as part of the partition if and only if the element point is also selected in the first triangle of the path. Here, the element point belongs to the southwest path.}\label{fig:alttriangles}
\end{figure}

\subsection{Alternating Triangle Paths}\label{subsec:alt}

It is shown in \cite{sw-96} that a partition of the points of $P_1$, $P_2$, and $P_3$ into triangles with total circumference $12n$ is possible if and only if a similar graph is constructed from a YES-instance of P3DM. We now outline this proof for our graph $G_{\text{TP}}$, as the geometric connection from P3DM to UC3P depends on the same pattern. 


\begin{pr} Given any instance of P3DM, construct a graph $G_{\text{TP}}$. There exists a pattern of alternating triangles in $G_{\text{TP}}$ if and only if P3DM has answer YES. \end{pr}

\begin{proof}
Consider a graph $G_{\text{TP}}$, with vertex sets $P_1$, $P_2$, and $P_3$, constructed as described in Section \ref{subsec:graph} from any instance of P3DM, and containing $n$ total 3-4-5 triangles. 

Suppose there exists a partition of $P_1 \times P_2 \times P_3$ into triangles with total circumference of $12n$. Each element point must be assigned to the 3-4-5 triangle at the start of one triangle path \cite{sw-96}. The next 3-4-5 triangle in the triangle path, the one in a rectangular pair with the first 3-4-5 triangle, must not be in the partition. Then the third 3-4-5 triangle must be included, since that is the only available minimum circumference triangle that includes the last vertex of the first pair of 3-4-5 triangles. This pattern propagates through the triangle path.

The first 3-4-5 triangle of any other paths originating at the same element point must not be selected in the partition, since the element point has already been assigned to a triangle. Therefore, the second 3-4-5 triangle in those triangle paths must be in the partition. Then the third 3-4-5 triangle must not be included in the partition, with this pattern again repeating through the triangle path.

This creates a particular assignment pattern in each triangle path, which we call an {\em alternating triangle path}, see Figure \ref{fig:alttriangles}. In the figure, the element point $(e)$ belongs to the 3-4-5 triangle at the start of the southwest path. When every triangle path is an alternating triangle path, we say a {\em pattern of alternating triangles} exists in graph $G_{\text{TP}}$.

Every path contains an even number of triangles. Therefore, if a triple triangle contains a vertex at which an alternating triangle path ends, where the first triple triangle of the path is included in the partition, the triple triangle must also be included in the partition. Conversely, if a triple triangle contains a vertex at which an alternating triangle path ends, where the second triple triangle is included in the partition, the triple triangle must not be included in the partition. The resulting pattern of alternating triangles creates a correspondence between element points (elements of $X$, $Y$, and $Z$), and the triple triangles (triples in $T$). Because the existence of a partition of total circumference $12n$ is assumed, each element point is assigned to one triple triangle, and the set of triples containing the selected triple triangles is a YES solution to P3DM. 

For the opposite implication, when P3DM has answer YES, a pattern of alternating triangles can be constructed using $T'$. Therefore MCM also has answer YES.
 \end{proof}

The primary focus in our proof of complexity for UC3P is proving that a combination measure which satisfies the requirements of UC3P when the input is constructed from a graph $G_{\text{TP}}$ also creates a pattern of alternating triangles when $N$ and $\Phi$ are not too far from minimum sparsity and optimal transport cost. Next, we describe how $G_{\text{TP}}$ serves as input for UC3P.

\subsection{Input for UC3P}\label{subsec:G2toBP}

The problem UC3P takes as input three probabilities measures with equal size support sets and uniformly distributed mass.
We use the vertex sets $P_1$, $P_2$, and $P_3$ from $G_{\text{TP}}$ for the support sets of probability measures; the edges of $G_{\text{TP}}$ are not used in UC3P. Recalling that some points appear twice in sets $P_1$, $P_2$, and $P_3$, the transformation to probability measures requires a bit of care, as the properties for discrete barycenters in \cite{abm-16} rely on an underlying assumption that each support point is unique. 

We allow support points within a measure to have the same coordinates, resulting in uniform mass assignment $1/n$ on all support points. If two or more support points in one measure have the same coordinates, we regain the usual form, that is, unique coordinates in the support set, by allowing some support points to have additional mass. Due to the structure of $G_{\text{TP}}$, no coordinate is repeated more than twice, so that no mass other than $1/n$ and $2/n$ would occur in a statement of the probability measures with unique coordinates. We choose to represent with duplicate coordinates, so that the resulting measures all have exactly $n$ support points. This assumption in the definition of UC3P makes the arguments in Section \ref{sec:HardProof} a bit less technical.

\section{A Barycenter of UC3P with $G_{\text{TP}}$ Input}\label{sec:HardProof}

Throughout this section, we consider instances $\mathcal{U}$ of UC3P, with fixed sparsity bound $N$ and transport cost $\Phi$, that are associated with an instance of P3DM; that is, the input measures $P_1, P_2$, and $P_3$ for UC3P are constructed from a graph $G_{\text{TP}}$ from the instance of P3DM.

We begin this section with a description of a formula for the cost $c(t_j)$ of a single triple $t_j = (\x_1^j,\x_2^j,\x_3^j)$. Then we examine $G_{\text{TP}}$ for the minimum cost $c(t_j)$ of any triple; in turn, this cost and the structure of $G_{\text{TP}}$ allows the determination of the transport cost of a barycenter of $P_1, P_2$, $P_3$. We call an arbitrary triangle containing vertices from $G_{\text{TP}}$ a {\em feasible triple} if each of its vertices belongs to a different set $P_1$, $P_2$, $P_3$.

In the discrete setting, the squared Wasserstein distance is the total squared Euclidean distance between the support points of the combination measure and the support points of $P_1, \ldots, P_m$. For each tuple $s_j \in S^*$, the transport cost of a unit of mass is \cite{abm-16,bp-18}: \begin{equation}\label{eqn:origcost} c(s_j) = \sum_{i=1}^m \lambda_i ||\x^j - \x_i^j||^2. \end{equation} 

For measures provided as $P^*$, Equation (\ref{eqn:origcost}) is not an ideal formula for the cost associated with a tuple, as there are two steps of calculation: first, $\x^j = \sum_{i=1}^m \lambda_i \x_i^j$, which is not part of the input, and then $c(s_j)$. Instead, a formula directly using the distances eliminates the need to compute $\x^j$ and simplifies the arguments to follow. Specifically, we consider the formula \begin{equation}\label{eqn:newcost} c(s_j) = \sum_{i=1}^{m-1} \lambda_i \sum_{k=i+1}^m \lambda_k ||\x_k^j - \x_i^j||^2.\end{equation} 
We show, through simple algebraic manipulation, Equation (\ref{eqn:newcost}) equals Equation (\ref{eqn:origcost}). By symmetry, and then through repeated use of the assumption $\sum_{i=1}^m \lambda_i = 1$, we have
\begin{align*}&\sum_{i=1}^{m-1} \lambda_i \sum_{k=i+1}^m \lambda_k || \x_k^j - \x_i^j ||^2 \text{, which is $c(s_j)$ of Equation (\ref{eqn:newcost}),}\\
=&\frac{1}{2} \sum_{i=1}^m \lambda_i \sum_{k=1}^m \lambda_k || \x_k^j - \x_i^j ||^2 \\=& \frac{1}{2}  \sum_{i=1}^m \lambda_i  (\sum_{k=1}^m \lambda_k (\x_k^j)^T \x_k^j- 2(\x_i^j)^T \x^j + (\x_i^j)^T \x_i^j) \\ 
=& \frac{1}{2} \sum_{i=1}^m \lambda_i \sum_{k=1}^m \lambda_k (\x_k^j)^T \x_k^j-  \sum_{i=1}^m \lambda_i (\x_i^j)^T \x^j + \frac{1}{2}\sum_{i=1}^m \lambda_i  (\x_i^j)^T \x_i^j \\=& \sum_{i=1}^m \lambda_i (\x_i^j)^T \x_i - (\x^j)^T \x^j \\
=& \sum_{i=1}^m \lambda_i (\x^j)^T \x^j -2 \sum_{i=1}^m \lambda_i \x_i^T \x^j + \sum_{i=1}^m \lambda_i (\x_i^j)^T \x_i^j \\
=& \sum_{i=1}^m \lambda_i ((\x^j)^T \x^j -2(\x_i^j)^T \x^j + (\x_i^j)^T \x_i^j) \\
 =& \sum_{i=1}^m \lambda_i ||\x^j - \x_i^j||^2 \text {, which is $c(s_j)$ of Equation (\ref{eqn:origcost}).}
\end{align*} 
Therefore Equation (\ref{eqn:newcost}) can be used as an alternative to Equation (\ref{eqn:origcost}).

Since UC3P has exactly three input measures, the tuples $s_j$ in $S^*$ can be described as triples $t_j = (\x_1^j,\x_2^j, \x_3^j)$ with $\x_1 \in \supp(P_1)$,  $\x_2 \in \supp(P_2)$, and  $\x_3 \in \supp(P_3)$.  Each triple has associated transport cost \begin{equation*} c(t_j) = 1/9( ||\x_{2}^j-\x_{1}^j||^2+||\x_{3}^j-\x_{1}^j||^2+||\x_{3}^j-\x_{2}^j||^2).\end{equation*} 

Additionally, let $S_{\scalebox{0.45}{\hspace{-0.04in}$P^*$}}^*$ be the set of indices $j$ of elements $s_j \in S^*$ such that $\x^j$ is in $\supp(P^*)$. Then the transport cost associated with $P^*$ is
\begin{equation}\label{eqn:phi} \phi(P^*) = \sum_{j\in S_{\scalebox{0.45}{\hspace{-0.04in}$P^*$}}^*} c(s_j) w_j. \end{equation} For UC3P, the total cost function of a combination measure is \begin{equation*} \phi( P^*) = 1/9\sum_{j\in S_{\scalebox{0.45}{\hspace{-0.04in}$P^*$}}^*} c(t_j) w_j =1/9\sum_{j\in S_{\scalebox{0.45}{\hspace{-0.04in}$P^*$}}^*} ( ||\x_{2}^j-\x_{1}^j||^2+||\x_{3}^j-\x_{1}^j||^2+||\x_{3}^j-\x_{2}^j||^2) w_j.  \end{equation*}

With these formulas, the costs $c(t_j)$ of any triple $t_j$ can be calculated using only the distances between the points of the triple, and without computing the weighted mean $\x_j$. We are now ready to prove Lemma \ref{lm:mintrip}. \\

\setcounter{repc}{0}
\begin{lm2}[Minimum Cost Triple]
For any instance $\mathcal{U}$, the minimum cost among all possible triples $t_j$ is $c(t_j) = 50/9$.
\end{lm2}

\begin{proof}
We prove the claim through an examination of all low-cost triples in $G_{\text{TP}}$. We first consider triples with all three vertices in the same triangle path, as shown in Figure \ref{fig:pathpattern}.

By construction, the distance between any two vertices of $G_{\text{TP}}$ is at least 3. Using this side length, we consider two triangles: a right triangle with two legs of length 3, or an improper triangle with three points on a line (Figure \ref{fig:side3options}). However, a right triangle with both legs of length 3 always has two points in the same measure; see Figure \ref{fig:side3options} (left). Thus, no feasible triples with two legs of length 3 exist. The improper triangle can indeed appear in a triangle path, and has associated cost $1/9 (9+9+36) = 54/9 = 6$, shown in Figure \ref{fig:side3options} (right). Any improper triangles with lengths greater than 3 will have a strictly greater associated cost; for an example, see Figure \ref{fig:stut} (left).

Since the right triangle with both legs of length 3 is never feasible, we now consider the cost of the primary 3-4-5 triangle used throughout the graph, as shown in Figure \ref{fig:stut} (right). The cost associated with this triple is $1/9 (9+16+25) = 50/9$, less than the cost of the improper triangle of Figure \ref{fig:side3options} (right). 

Within a triangle path, any other triples contain two sides of at least 6 units. When a triangle has at least two sides of length 6, the corresponding cost exceeds $1/9(36+36+9) = 9$, and is always greater than the cost of the 3-4-5 triangle, $50/9$. Therefore the 3-4-5 triangle is the lowest cost triple within a single triangle path.

Now consider a (feasible) triple with at least two vertices that do not belong to the same triangle path. Then one of two cases hold:
\begin{enumerate}
\item The triple contains an element point. Then by construction, the lowest cost possible triple is a 3-4-5 triangle, with cost $50/9$, as before. Such triples do exist; for an example see the lower right of the element point in Figure \ref{fig:threepath}.
\item The triple does not contain an element point. By construction, two vertices in the triple must be at least 6 units apart, and the associated cost must exceed that of the 3-4-5 triangle. 
\end{enumerate}

Having exhausted all possibilities, triples $t_j$ associated with a 3-4-5 triangle have the lowest possible cost, $c(t_j) = 50/9$.
 \end{proof}

\begin{figure}[!t]
\begin{center}
\begin{tikzpicture}[scale=0.5]

\draw [thick](-2,6)--(-2,9) node [midway, left]{$3$};
\draw [thick](-2,6)--(1,6) node [midway, below]{$3$};
\draw[thick] (-2,9)--(1,6);

\fill [cyan] (1, 6) circle (6pt) node [anchor = north west]{$P_3$};
\fill [red] (-2,6) circle (6pt) node [anchor = north east]{$P_2$};
\fill [cyan] (-2,9) circle (6pt) node [anchor = south east]{$P_3$};
\end{tikzpicture}\qquad
\begin{tikzpicture}[scale=0.5]

\draw [thick](-2,6)--(-5,6) node [midway, below]{$3$};
\draw [thick](-2,6)--(1,6) node [midway, below]{$3$};
\fill [cyan] (1, 6) circle (6pt) node [anchor = north west]{$P_3$};
\fill [red] (-2,6) circle (6pt) node [anchor = north]{$P_2$};
\fill [blue] (-5,6) circle (6pt) node [anchor = north east]{$P_1$};
\end{tikzpicture}

\end{center}
\caption{(left) All occurrences of a right triangle with two legs of length 3 have two same-colored points. Therefore no such triples are feasible. (right)  An improper triangle with side lengths 3, 3, and 6. The total cost of such a triple is $1/9 (9+9+36) = 54/9 = 6$.}\label{fig:side3options}
\end{figure}
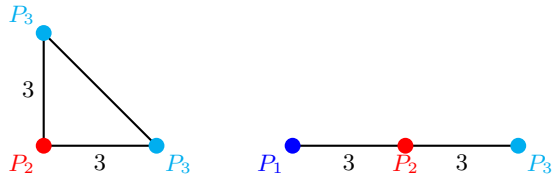

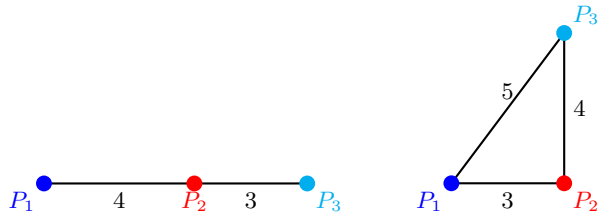
\begin{figure}
\begin{center}
\begin{tikzpicture}[scale=0.5]
\draw [thick](-2,6)--(-6,6) node [midway, below]{$4$};
\draw [thick](-2,6)--(1,6) node [midway, below]{$3$};
\fill [cyan] (1, 6) circle (6pt) node [anchor = north west]{$P_3$};
\fill [red] (-2,6) circle (6pt) node [anchor = north]{$P_2$};
\fill [blue] (-6,6) circle (6pt) node [anchor = north east]{$P_1$};

\end{tikzpicture}
\qquad
\begin{tikzpicture}[scale=0.5]

\draw [thick](1,2)--(-2,2) node [midway, below]{$3$};
\draw [thick] (1,2)--(1,6) node [midway, right]{$4$};
\draw [thick] (1,6)--(-2,2)node [midway, left, above]{$5$};
\fill [red] (1,2) circle (6pt) node [anchor = north west]{$P_2$};
\fill [blue] (-2,2) circle (6pt) node [anchor = north east]{$P_1$};
\fill [cyan] (1, 6) circle (6pt) node [anchor = south west]{$P_3$};
\end{tikzpicture}
\end{center}
\caption{(left) Compared to the improper triangle of Figure \ref{fig:side3options} (right), the cost associated with this improper triangle is strictly worse. (right) The cost associated with this triple is $1/9 (9+16+25) = 50/9$, which is the minimum among all triples. }\label{fig:stut}
\end{figure}

Since a 3-4-5 triangle is always the preferred choice of triple due to its minimum cost, in the following, we simply refer to a 3-4-5 triangle as a triangle. The minimum cost of $c(t_j)$ from Lemma \ref{lm:mintrip} allows the construction of a barycenter that gives the following results.

\setcounter{repc}{3}
\begin{thm2}[Transport Cost of a Barycenter] For an instance $\mathcal{U_Y}$, a barycenter has transport cost $\phi(\bar P) = 50/9$. \end{thm2}


\setcounter{repc}{0}
\begin{cor2}[Existence of a Sparse Barycenter] For an instance $\mathcal{U_Y}$, there exists a barycenter with $n$ support points. \end{cor2}

\begin{proof}[{Proof of Theorem \ref{thm:baryvalue} and Corollary \ref{cor:barysparse}.}]  
Let an instance of P3DM have answer YES, and construct $G_{\text{TP}}$ and $P_1$, $P_2$, and $P_3$ as described in Section \ref{sec:G2}. In a YES-instance of P3DM, there exists a set $T'$ where each element appears in exactly one triple. Fix such a set $T'$.

For all element points, assign mass $1/n$ to the first triangle in the triangle path leading to the triple in which the element point is represented in $T'$. Proceed through the triangle path, assigning mass $1/n$ to alternating triangles, so that this is an alternating triangle path which includes the element point.

For all other paths (those leading from element points to triples not in $T'$), assign mass $1/n$ to the second triangle in the path, that is, the first triangle not containing the element point. Again, complete the alternating triangle path by assigning mass $1/n$ to alternating triangles. 

Such a mass assignment gives a measure $P^*$  that satisfies the mass transport requirements of all points in $P_1, P_2, P_3$. Furthermore, all mass is assigned to a triple $t_j$ of minimum cost $c(t_j) = 50/9$, so $\phi(P^*)$ is minimal. Since no lower transport cost is possible, $P^*$ is a barycenter, and has transport value:
\begin{equation*}
\phi(P^*) = \sum_{j=1}^{n} \frac{50}{9}( \frac{1}{n}) = \frac{50}{9} \sum_{j=1}^n \frac{1}{n}  =  \frac{50}{9}. 
\end{equation*}

This proves Theorem \ref{thm:baryvalue}. The measure $P^*$ contains $n$ support points, one for each triangle. Since it is a barycenter, the existence of measure $P^*$ gives Corollary \ref{cor:barysparse}.  \end{proof} 
\begin{figure}[t]
\begin{center}
\fbox{\begin{tikzpicture}[scale=0.5]


\draw [thick](1,2)--(-0.5,2);
\draw [thick] (1,2)--(1,4);
\draw [thick] (1,4)--(-0.5,2);

\draw [thick](-0.5,4)--(-3,4);
\draw [thick] (-0.5,4)--(-0.5,5.5);
\draw [thick](-0.5,4)--(1,4);
\draw [thick] (-0.5,4)--(-0.5,2);
\draw [thick] (-0.5,5.5)--(-3,4);

\draw [thick](-3,5.5)--(-0.5,5.5);
\draw [thick] (-3,5.5)--(-3,4);


\draw [thick](1,2)--(2.5,2);
\draw [thick] (2.5,2)--(2.5,4);
\draw [thick] (1,4)--(2.5,2);
\draw[thick] (1,4)--(2.5,4);

\draw[thick](4.5,4)--(2.5,4);
\draw[thick](2.5,4)--(2.5,5.5);
\draw[thick](4.5,4)--(2.5,5.5);
\draw[thick](4.5,4)--(4.5,5.5);
\draw[thick](2.5,5.5)--(4.5,5.5);

\fill [white] (-2.5,-2.5) circle;
\fill [red] (-0.5,4) circle (6pt) node [anchor = north east]{$P_2$};
\fill [red] (1,2) circle (6pt) node [anchor = north]{$P_2${\color{black}\tiny(e)}};
\fill [blue] (-0.5,2) circle (6pt) node [anchor = north east]{$P_1$};
\fill [cyan] (1, 4) circle (6pt) node [anchor = south]{$P_3${\tiny (2)}};
\fill [blue] (-3,4) circle (6pt) node [anchor = north east]{$P_1$};
\fill [cyan] (-0.5, 5.5) circle (6pt) node [anchor = south west]{$P_3$};
\fill [red] (-3,5.5) circle (6pt) node [anchor = south east]{$P_2$};
\fill [red] (2.5,4) circle (6pt) node [anchor = north west]{$P_2$};
\fill [blue] (2.5,2) circle (6pt) node [anchor = north west]{$P_1$};
\fill [blue] (4.5, 4) circle (6pt) node [anchor = north west]{$P_1$};
\fill [cyan] (2.5, 5.5) circle (6pt) node [anchor = south east]{$P_3$};
\fill [red] (4.5,5.5) circle (6pt) node [anchor = south west]{$P_2$};
\end{tikzpicture}}
\qquad
\fbox{\begin{tikzpicture}[scale=0.5]

\draw [thick](1,2)--(-0.5,2);
\draw [thick] (1,2)--(1,4);


\draw [thick](1,2)--(2.5,2);
\draw [thick] (2.5,2)--(2.5,4);
\draw [thick] (1,4)--(2.5,2);
\draw[thick] (1,4)--(2.5,4);

\draw[thick](4.5,4)--(2.5,4);
\draw[thick](2.5,4)--(2.5,5.5);
\draw[thick](4.5,4)--(2.5,5.5);
\draw[thick](4.5,4)--(4.5,5.5);
\draw[thick](2.5,5.5)--(4.5,5.5);


\draw[thick](1,2)--(1,0);
\draw[thick](1,0)--(-0.5,2);
\draw [thick] (1,0)--(-0.5,0);
\draw [thick] (-0.5,0)--(-0.5,2);

\draw[thick] (-2.5,-1.5)--(-0.5,-1.5);
\draw[thick] (-2.5,-1.5)--(-2.5,0);
\draw[thick] (-0.5,-1.5)--(-2.5,0);
\draw[thick] (-2.5,0)--(-0.5,0);
\draw[thick] (-0.5,-1.5)--(-0.5,0);

\fill [red] (2.5,4) circle (6pt) node [anchor = north west]{$P_2$};
\fill [blue] (2.5,2) circle (6pt) node [anchor = north west]{$P_1$};
\fill [blue] (4.5, 4) circle (6pt) node [anchor = north west]{$P_1$};
\fill [cyan] (2.5, 5.5) circle (6pt) node [anchor = south east]{$P_3$};
\fill [red] (4.5,5.5) circle (6pt) node [anchor = south west]{$P_2$};
\fill [cyan] (1,0) circle (6pt) node [anchor = north west]{$P_3$};
\fill[red] (-0.5,0) circle (6pt) node [anchor = north west]{$P_2$};
\fill [red] (1,2) circle (6pt) node [anchor = north west]{$P_2${\color{black}\tiny(e)}};
\fill [blue] (-0.5,2) circle (6pt) node [anchor = south east]{$P_1$};
\fill [cyan] (1, 4) circle (6pt) node [anchor = south east]{$P_3$};
\fill[blue] (-2.5,0) circle (6pt) node [anchor = south east]{$P_1$};
\fill[cyan] (-0.5,-1.5) circle (6pt) node [anchor = north west]{$P_3$};
\fill[red] (-2.5,-1.5) circle (6pt) node [anchor = north east]{$P_2$};
\end{tikzpicture}}

\end{center}
\caption{When two paths leave an element point, two possible configurations exist. (left) When the paths are adjacent, the element point does not belong to any minimum cost triangles that do not belong to a single path. (right) When the paths lie opposite, as in this configuration, the element point belongs to two minimum cost triangles that have other vertices from two paths. }\label{fig:twopath}
\end{figure}
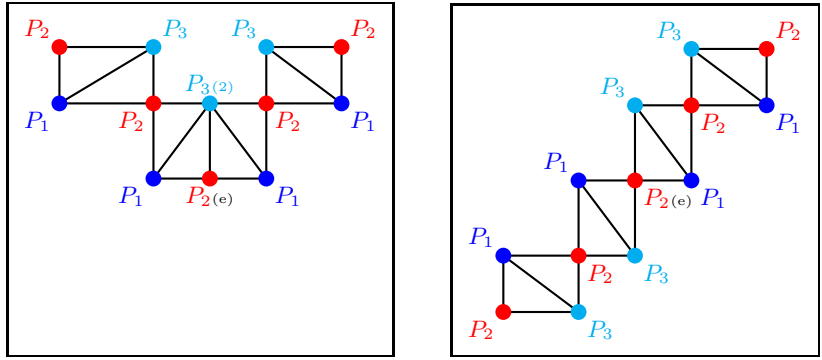

We now show that the combinations in a combination measure satisfying the requirements of UC3P in $\mathcal{U}$ create a pattern of alternating triangles.

\setcounter{repc}{4}
\begin{thm2}[Barycenters Create a Pattern of Alternating Triangles]
For an instance $\mathcal{U}$ that has a barycenter $P$ of $n$ support points and transport cost $50/9$, $P$ creates a pattern of alternating triangles. 
\end{thm2}

\begin{proof}
Let $P$ be a barycenter of $n$ support points and transport cost $50/9$ of an instance $\mathcal{U}$. Then each support point of $P$ must have mass $1/n$ in order to satisfy the non-mass-splitting property, and all corresponding triples in the transport plan must correspond to (3-4-5) triangles. We first show that for all element points, mass must be assigned to a triangle consisting of the element point and two vertices in the same triangle path; equivalently, for all element points, mass is assigned to the first triangle in exactly one triangle path. Then we show that this mass assignment creates alternating triangle paths.

Because $P$ has a total transport cost of $50/9$, each element point must be assigned to a triangle with cost $50/9$. In order for an element point to be a vertex of a minimum cost triangle, the element point must be incident to at least one triangle path. In P3DM, this corresponds to the requirement that every element must appear in at least one triple of $T'$, a necessary condition for a YES-instance. Thus every element point must be incident to one, two, or three triangle paths. Let $e$ be any element point.
\begin{enumerate} 
\item[\textbf{I.}] \textbf{Suppose exactly one triangle path starts at $e$.} Then mass $1/n$ must be assigned to the first triangle in the triangle path, as it is the only triple of minimum cost that contains the element point. 

\item[\textbf{II.}] \textbf{Suppose exactly two triangle paths start at $e$.} Then either: 
\begin{enumerate}
\item[1.] The two paths start with triangles which share the element point and a duplicate vertex (2); see Figure \ref{fig:twopath} (left). In this case, there are no (3-4-5) triangles with vertices in multiple paths, so mass $1/n$ must be assigned to the first triangle in exactly one of the paths. 
\item[2.] The two paths start on opposite corners of the element point and only share the element point; see Figure \ref{fig:twopath} (right). In this case, the element point belongs to four minimum cost triangles: two at the start of the two paths, and two triangles which contain the element point and one vertex from each of the two paths, hereafter called an off-path triangle. Suppose mass $1/n$ is assigned to an off-path triangle. 

Then mass cannot be assigned to either of the two triangles at the start of the paths, since the element point already has mass. Mass cannot be assigned to either of the second triangles in the paths, since the other two vertices of the selected off-path triangle have also been assigned mass. This leaves one vertex from the first triangle in each of the two paths without mass. 

These two vertices cannot be in the same triangle, as there is no available triangle of minimum cost containing both vertices. There are also no minimum cost triangles containing each vertex individually, using vertices from the third and fourth triangles in the path, due to the assumption that paths have horizontal orientation on the second pair of triangles. Therefore mass cannot be assigned to an off-path triangle, as doing so necessitates a mass assignment to a triple of cost higher than $50/9$. 
\end{enumerate} 

\item[\textbf{III.}]  \textbf{Suppose exactly three triangle paths start at the element point.} Then there is exactly one off-path triangle; recall Figure \ref{fig:threepath}. Suppose mass is assigned to the off-path triangle. 

Then, as in Case II.2, there are two vertices from the first triangle in two triangle paths that still require mass. Both of these vertices must be duplicates; that is, duplicated in the support sets of their $P_i$, and labeled (2) in Figure \ref{fig:threepath}. One minimum cost triangle exists with as yet unassigned mass for these two vertices: the second triangle in the triangle path whose vertices, other than the element point, are not yet assigned to a triple. However, by the same argument as in II.2, there is no other minimum cost triangle available for these two points.
\end{enumerate}

\noindent Having examined all possible path configurations in I, II, and III, when there exists a minimum sparsity barycenter with transport cost $50/9$, mass must be assigned to the first triangle in exactly one triangle path at each element point.

Next, consider an element point $e$, and let mass be assigned to the first triangle in a triangle path at $e$. There is only one case where the mass transport requirements of $P_1$, $P_2$, and $P_3$ allow for mass to be assigned to the second triangle in that path: when the element point $e$ belongs to the first triangle in three triangle paths, and the other two vertices of the first triangle are both duplicates (essentially, mass is assigned to the first triangle of the ``middle path''). As in II.2, two vertices from the first triangle in the other two triangle paths have no available minimum cost triangle. Therefore in all cases, mass cannot be assigned to the second triangle in a triangle path when the first triangle is assigned mass. Mass must be assigned to the third triangle in the triangle path.

Any other triangle paths which begin at $e$ cannot have mass assigned to the first triangle, since mass transport to the element point has been satisfied. Then the only available minimum cost triangle for the other two vertices in the first triangle is the second triangle in the triangle path, so mass must be assigned to the second triangle. 

Other off-path triangles may occur along each path; however, mass cannot be assigned to them, following the same arguments as in II.2: assigning mass to an off-path triangle necessitates the inclusion of a triangle with cost greater than $50/9$. Therefore all paths are alternating triangle paths, and the combinations from $P$ have produced a pattern of alternating triangles. 

Therefore, if there exists a minimum sparsity barycenter with transport cost $50/9$, the combinations contained in the barycenter create a pattern of alternating triangles in graph $G_{\text{TP}}$.  \end{proof}



The reduction of P3DM to UC3P follows from the previous results.

\begin{thm2}[P3DM Reduction to UC3P] For an instance $\mathcal{U}$ with $N = n$ and $\Phi = 50/9$, UC3P has answer YES if and only if P3DM has answer YES.
\end{thm2}

\begin{proof}
$(\Rightarrow)$ Let an instance $\mathcal{U}$ have a barycenter with $n$ support points and $50/9$, so that UC3P has answer YES for $N \geq n$ and $\Phi \geq 50/9$. By Theorem \ref{thm:sparsealt}, this barycenter must create a pattern of alternating triangles. By Proposition \ref{prop:P3DMalt}, P3DM must also have answer YES. 

$(\Leftarrow)$ Suppose P3DM has answer YES, construct a corresponding graph $G_{\text{TP}}$ and instance $\mathcal{U_Y}$. Then by Theorem \ref{thm:baryvalue} and Corollary \ref{cor:barysparse}, a barycenter for UC3P with input measures $P_1, P_2, P_3$ from $G_{\text{TP}}$ has value $50/9$, and at least one barycenter exists with $n$ support points. So for $N \geq n$ and $\Phi \geq 50/9$, UC3P has answer YES.
 \end{proof}

SCMP is immediately known to be NP-hard, since UC3P is NP-hard. The result that UC3P is NP-hard also implies the following.

\setcounter{repc}{2}
\begin{thm2}
For $m\geq 3$ and $d \geq 2$, an efficient algorithm for SBP cannot exist, unless P = NP. 
\end{thm2}
\begin{proof}
Suppose an efficient algorithm for solving SBP exists, and consider SBP for the measures $P_1$, $P_2$, and $P_3$ from the graph $G_{\text{TP}}$ of an instance of P3DM. Then for a given sparsity $N$, an efficient decision on the existence of a barycenter with sparsity at most $N$ can be made by simply computing a barycenter. Following the same arguments regarding the creation of a pattern of alternating triangles, SBP has a barycenter of sparsity at most $N$ if and only if P3DM has answer YES. Thus, if an efficient algorithm for SBP exists, then an efficient decision for P3DM can be made. So an efficient algorithm for SBP must not exist, unless P = NP.
 \end{proof}

\section{Challenges of Proving Containment in NP and Verifying Optimality}\label{sec:open}

In this paper, we proved that UC3P is NP-hard (see Sections \ref{sec:proofstrat} to \ref{sec:HardProof}). In doing so, we saw that the more general SCMP is NP-hard, too, and that there cannot exist an efficient algorithm for the sparse barycenter problem (SBP), unless $P=NP$. 

We now turn to the question ``Is SCMP NP-Complete?'' and the corresponding implications for SBP. In Section \ref{subsec:SCMPNP}, we examine the conditions for an efficient verification of a combination measure provided for SCMP. Then we extend the discussion to verifying a barycenter in Section \ref{subsec:SBPver}. The efficiency of any verification is determined relative to the encoding size of ``the input'', which refers to  different measures depending on the context. We use the phrasing ``relative to'' when specifying the complexity or ``efficient relative to'' or when specifying the {\em polynomial} complexity with respect to the encoding size of a measure. As we require knowledge of the efficiency with respect to the encoding size of the {\em original} input measures $P_1, \ldots, P_m$, for convenience, we let $L$ represent the largest encoding size of $P_1, \ldots, P_m$. Recall that we assume that the encoding size of the original input is dominated by the size of the measures (and not $N$ or $\lambda$). In particular, it is linear in $L$.




\subsection{Challenges of Proving SCMP is in NP}\label{subsec:SCMPNP}

To verify a combination measure satisfies the requirements of SCMP, we need to check that a given measure has a support set of size below bound $N$ and a non-mass-splitting transport cost below bound $\Phi$. For any measure (even if it is not a combination measure), checking the sparsity -- that is, the size of its support set -- is trivial. Therefore the sparsity requirement of SCMP is not problematic for determining if SCMP is NP. 


So let us consider the computation of the transport cost $\phi(P^*)$, as verifying that the transport cost is below the bound $\Phi$ implies the computation of the transport cost itself. 

\begin{lemma}\label{lm:phiPstar} Let $P_1, \ldots, P_m$ be measures with support points in $\Q^d$, masses in $\Q$, and maximum encoding size $L$. Let  $\lambda $ be a corresponding weight vector in $\Q^m$. Let $P^*$ be an associated combination measure with masses in $\Q$.

Then the transport cost $\phi(P^*)$ can be computed efficiently relative to $P^*$. When $P^*$ has at most $N$ support points and encoding size polynomial in $L$, then $\phi(P^*)$ can be computed efficiently relative to $P_1, \ldots, P_m$. \end{lemma}

\begin{proof} The formula for $\phi(P^*)$ is given in Equation (\ref{eqn:phi}) and takes as input the transport costs $c$ and the masses associated with the tuples of $P^*$. In Section \ref{sec:HardProof}, we give two formulas for $c$ in Equation (\ref{eqn:origcost}) and Equation (\ref{eqn:newcost}). We show that both strategies are efficient.

First consider Equation (\ref{eqn:origcost}). The weighted means, one for each tuple of $P^*$, can be computed efficiently relative to $P^*$. Then, for each tuple, the weighted distances from the weighted mean to each of the points in the tuple are computed and totaled; this requires $|P^*| \cdot m$ weighted distance calculations. This is efficient relative to $P^*$. 

Now consider Equation (\ref{eqn:newcost}). We compute the squared Euclidean distances between the elements of the tuple, which are support points in $P_1, \ldots, P_m$. Since the support points are rational, so too are the resulting distances. Each tuple requires $\frac{m(m-1)}{2}$ distance computations, so the total number of computations is $|P^*|\cdot (\frac{m(m-1)}{2})$, efficient relative to $P^*$.

Then in Equation (\ref{eqn:phi}), the costs $c$ are multiplied by the masses associated with each tuple; since the masses are rational by assumption, this computation is efficient relative to $P^*$. 

Now suppose $P^*$ is sparse. In Equation (\ref{eqn:origcost}), the number of weighted means computations $N$ and the resulting number of weighted distance computations $Nm$ are efficient relative to $P_1, \ldots, P_m$. Using Equation (\ref{eqn:newcost}), the total number of computations is now $N (\frac{m(m-1)}{2})$, which is efficient relative to $P_1, \ldots, P_m$. 

Finally, when $P^*$ has encoding size polynomial in $L$, the computation of $\phi(P^*)$ is efficient relative to $P_1, \ldots, P_m$.
 \end{proof}

%
%
%
%

Lemma \ref{lm:phiPstar} implies that a proposed solution $P^*$ for SCMP with sufficiently small encoding size can be verified efficiently.

\begin{theorem}\label{thm:effSmallRat} A combination measure $P^*$ supplied as a possible solution to SCMP can be verified efficiently if the measure has an encoding size polynomial in $L$.
\end{theorem}

\begin{proof}
Since the measure is provided as a combination measure, the non-mass-split of the transport is a given. Thus, the efficient verification of the transport cost is an immediate consequence of Lemma \ref{lm:phiPstar}, since the sparsity check is trivial and all that remains is to compare the efficiently computed $\phi(P^*)$ to the bound $\Phi$. 
 \end{proof}

If all sparse combination measures would satisfy the encoding size assumption of Lemma \ref{lm:phiPstar}, then we could conclude that SCMP is in NP. However, it remains to address the possible mass assignments. Even though the input masses are in $\Q$, there may exist (sparse) barycenters with rational masses exceeding any given bound on the size of the encoding, and even with irrational mass. We exhibit a simple example.

\begin{example}\label{ex:1} Suppose we have two measures with at least two support points that lie on the vertices of a square, and any other support points sufficiently far away that the optimal transport contains tuples with support points from multiple vertices of the square. Such a configuration is shown in Figure \ref{fig:irrbary}. The support points have equal mass $d$ (if there are no other support points in the measures, $d = 1/2$). 

\begin{figure}[t]
\begin{center}
\begin{tikzpicture}[scale=0.5]

\fill [red] (6,6) circle (6pt) node [anchor = south west]{$P_1$};
\fill [blue] (0,6) circle (6pt) node [anchor = south east]{$P_2$};
\fill [blue] (6,0) circle (6pt) node [anchor = north west]{$P_2$};
\fill [red]  (0,0) circle (6pt) node [anchor = north east]{$P_1$};

\fill (3,0) circle (5pt) node [anchor = north]{$\bar P$};
\fill (3,6) circle (5pt) node [anchor = south]{$\bar P$};
\fill (0,3) circle (5pt) node [anchor = east]{$\bar P$};
\fill (6,3) circle (5pt) node [anchor = west]{$\bar P$};

\end{tikzpicture}\qquad\qquad\qquad\qquad
\begin{tikzpicture}[scale=0.5]

\fill [red] (6,6) circle (6pt) node [anchor = south west]{$d$};
\fill [blue] (0,6) circle (6pt) node [anchor = south east]{$d$};
\fill [blue] (6,0) circle (6pt) node [anchor = north west]{$d$};
\fill [red] (0,0) circle (6pt) node [anchor = north east]{$d$};

\fill (3,0) circle (5pt) node [anchor = north]{$b$};
\fill (3,6) circle (5pt) node [anchor = south]{$b$};
\fill (0,3) circle (5pt) node [anchor = east]{$d-b$};
\fill (6,3) circle (5pt) node [anchor = west]{$d-b$};

\draw[thick,->] (0,3)--(0,0.15);
\draw[thick,->] (0,3)--(0,5.85);
\draw[thick,->] (6,3)--(6,0.15);
\draw[thick,->] (6,3)--(6,5.85);
\draw[thick,->] (3,0)--(0.15,0);
\draw[thick,->] (3,0)--(5.85,0);
\draw[thick,->] (3,6)--(0.15,6);
\draw[thick,->] (3,6)--(5.85,6);

\end{tikzpicture}
\end{center}
\caption{(left) Two measures $P_1$ and $P_2$ with $|P_1| = |P_2| = 2$ and uniformly distributed mass. The measure $\bar P$ has four possible support points. (right) $\bar P$ is a barycenter for this input for any value $b \in [0,d]$.}\label{fig:irrbary}
\end{figure}
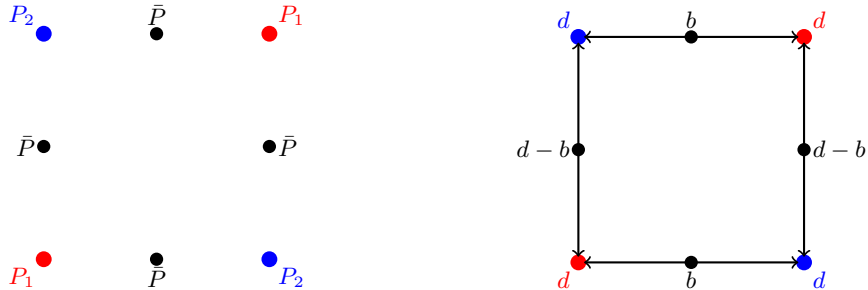

Then four weighted means with identical minimum costs $c$ exist, and any balanced mass assignment, that is, $b \in [0,d] \in \mathbb{R}$, has the same minimum total transport cost. This shows that there exist rational barycenters of arbitrary encoding size (and, less noteworthy, even irrational barycenters).  
\end{example} 

This example is easily extended to three or more measures to fit the requirements of UC3P and SCMP. Recall that the barycenter problem can be modeled and solved through various exponentially-scaling linear programs \cite{bp-18}. Situations like in Example \ref{ex:1} arise when the optimal face of the underlying polyhedron is of dimension one or higher. The vertices of this optimal face have a smaller support and there are guaranteed upper bounds on the size of a bit encoding of the vertices of a polyhedron (linear programming is known to be in NP). However, due to the exponential scaling of the linear programs (depending on $S^*$ or $S$ and independently of the measure provided as a possible barycenter), a linear programming-based approach does not give an efficient way to transition to a measure of such better properties. Since the encoding size of solutions to SCMP could be arbitrarily large relative to the encoding size of the original input, it remains open whether SCMP is in NP.

\subsection{Further Challenges for SBP}\label{subsec:SBPver}

Recall that all optimal transport plans for barycenters are non-mass-splitting, so all barycenters can be represented as combination measures. Thus the challenge of verifying a solution to SCMP is also an immediate challenge for verifying a solution to SBP, as a combination measure's encoding size may be arbitrarily large relative to $L$. 

We now highlight an additional challenge for verifying a solution to SBP: it remains open whether one can efficiently verify that a given barycenter, in fact, is a barycenter. To do so, one would have to rule out the existence of a measure with strictly better transport cost. As in the discussion above, at first glance, the various linear programs to find a barycenter seem like a promising tool. Barycenters lie on the boundary of the underlying polytopes and sparsest barycenters are vertices. For general linear programs, there are several ways to efficiently verify whether a given point is optimal through some simple algebra, for example through checking whether the objective function vector lies in the cone of outer normals of the point or whether a corresponding dual solution has the same value. However, the same issue persists: the scaling of these linear programs depends on the sizes of $S$ and $S^*$, which generally can be of exponential size -- independently of the sparsity of the given measure. Additionally, at this time it is even open whether one can efficiently verify that a given measure is contained in the set of uniqued weighted means without some extra information. 

However, if we can find an optimal transport plan of a measure and determine it is mass-splitting, we can eliminate that measure as a potential barycenter. We now examine the efficiency of computing an optimal transport plan since barycenters are often provided without a transport plan; that is, we assume we have a measure $P$ given as a set of support points and corresponding masses. We note that if a non-mass-splitting transport plan is also given or can be computed for any subset of $S$, it is trivial to go to representation $P^*$.

An optimal transport plan for $P$ can be determined through a linear program of \cite{abm-16}. In the linear program, the objective function uses the {\em weighted distances} from each support point $\x^j$ of $P$ and every support point $\x_i$ in $P_1, \ldots, P_m$. That is, we must calculate 
\begin{equation}\label{eqn:weighteddist} \lambda_i || \x^j - \x_i||^2\end{equation} for each $\x^j$ and all $\x_i$ in the support of $P_i$, $i = 1, \ldots, m$. The following lemma states that computing the necessary weighted distances for the linear program is efficient relative to the measure $P$; recalling the bound $N$ on the number of support points of a `sparse' measure, the computation is efficient relative to $P_1, \ldots, P_m$ when $P$ is sparse. 

\begin{lemma}\label{lm:sparsec} Let $P_1, \ldots, P_m$ be measures with support points in $\Q^d$ and masses in $\Q$, and a corresponding weight vector $\lambda \in \Q^m$. Let $P$ be another measure with support points in $\Q^d$. 

The weighted distances from the support points of $P$ to all support points of  $P_1, \ldots, P_m$ can be computed efficiently relative to $P$ and $P_1, \ldots, P_m$. If $P$ has at most $N$ support points and encoding size polynomial in $L$, the weighted distances from $P$ to $P_1, \ldots, P_m$ can be computed efficiently relative to $P_1, \ldots, P_m$. \end{lemma}


\begin{proof}
We first note that because $P$ has support points $\x^j \in \Q^d$, the weighted distances (see (\ref{eqn:weighteddist})) are in $\Q$ due to the use of the squared Euclidean distance. 

For each support point of $P$, the total number of weighted distance calculations is the total number of support points of $P_1,\ldots, P_m$: at most $mn$. Therefore there are at most $|P| \cdot mn$ weighted distance calculations, which is polynomial with respect to the size of $P$. 

When $P$ is sparse, that is, has at most $N$ support points, the number of required weighted distance calculations $Nmn$ is polynomial relative to $P_1, \ldots, P_m$. 
 \end{proof}

Once the weighted distances have been computed, we have the necessary information to set up the LP of \cite{abm-16}. Under the assumption that the measure has already been verified to be sparse, the resulting LP is polynomially-sized and  produces an optimal transport plan. 


\begin{theorem}\label{thm:opttrans} Let $P_1, \ldots, P_m$ be measures with support points in $\Q^d$, masses in $\Q$, and maximum encoding size $L$. Let a corresponding weight vector $\lambda$ be in $\Q^m$. Let $P$ be a probability measure with support points in $\Q^d$ and masses in $\Q$. 

By solving a polynomially-sized linear program, an optimal transport plan for $P$ can be computed in strongly polynomial time with respect to the size of $P$. If $P$ has at most $N$ support points and an encoding size polynomial in $L$, an optimal transport plan can be computed efficiently with respect to the size of $P_1, \ldots, P_m$. \end{theorem}

\begin{proof}
The program formulation of \cite{abm-16} requires at most $|P| \cdot mn$ variables and $|P|+mn$ constraints. Therefore the linear program is polynomial in size with respect to the size of $P$. 

The coefficients of the objective function of the linear program of \cite{abm-16} are the weighted distances from the support points of $P$ to each of the support points of $P_1, \ldots, P_m$. By Lemma \ref{lm:sparsec}, these costs can be computed efficiently. So the objective function can be constructed in a polynomial number of operations relative to the size of $P$. 

Furthermore, all constants in the program are rational, including those appearing in the constraints, since the masses of $P$ are assumed to be rational. The resulting linear program can be solved in strongly polynomial time due to the famous result in \cite{t-86} and the fact that all coefficients in the constraint matrix are $0$ or $1$. 

When $P$ is sparse, that is, has at most $N$ support points, the number of weighted distances and variables is $kmn$, and the number of constraints is $k+mn$. When the encoding size of $P$ is polynomial in $L$, specifically, the encoding size of the rational constants in the constraints of the LP, we produce an LP solvable efficiently relative to $P_1, \ldots, P_m$. 
\end{proof}

If the computed optimal transport plan for a provided measure is mass-splitting, the measure is not a barycenter. The verification of non-mass-split requires a simple examination of the optimal values of the $|P| \cdot mn$ variables, so such a verification is efficient relative to $P_1, \ldots, P_m$. Theorem \ref{thm:opttrans} also gives that for a barycenter with masses in $\Q$, it is efficient to transform between representations $P$ and $P^*$ relative to the barycenter.

\begin{lemma}
Let $P_1, \ldots, P_m$ be measures with support points in $\Q^d$, masses in $\Q$, and maximum encoding size $L$. Let  $\lambda$ be a corresponding weight vector in $\Q^m$. Let $P$ and $P^*$ be representations of the same barycenter for measures $P_1, \ldots, P_m$ with masses in $\Q$. 

The transformation between representations $P$ and $P^*$ is efficient relative to $P$ and $P^*$. When the represented barycenter has at most $N$ support points and encoding size polynomial in $L$, the transformation is efficient relative to $P_1, \ldots, P_m$.
\end{lemma}

\begin{proof}
Suppose $P$ and $P^*$ represent a barycenter with masses in $\Q$. We can transform $P^*$ to $P$ efficiently relative to $P^*$, so suppose we wish to transform $P$ into $P^*$. It is known that all support points are weighted means of points from $P_1, \ldots, P_m$ and thus necessarily rational. Therefore an optimal transport plan may be computed efficiently by Theorem \ref{thm:opttrans}. Further, all optimal transport plans for a barycenter are non-mass-splitting, so the computed transport plan is necessarily non-mass-splitting. Therefore, Theorem \ref{thm:opttrans} implies that it is efficient to transform $P$ to $P^*$ if $P$ is a barycenter with rational masses. 

Thus, it is efficient to transform between representations for a barycenter with respect to the size of the barycenter. By the same arguments, $P$ to $P^*$ is efficient relative to $P_1, \ldots, P_m$ when the sparsity and encoding size requirements of Theorem \ref{thm:opttrans} are satisfied. 
 \end{proof}

Thus when verifying a solution to SBP, a barycenter may be provided in either representation. This does not apply to SCMP: the ability to efficiently compute an optimal transport to $P_1,\ldots, P_m$ does not imply that the transformation of representation $P$ to representation $P^*$ is efficient for {\em arbitrary} measures $P$: even if there exists an optimal non-mass-splitting transport from $P$ to $P_1, \ldots, P_m$, the optimal transport that is computed may still be mass-splitting when $P$ is not a barycenter; even if all the support points of $P$ are weighted means in $S$, without a non-mass-splitting transport plan the measure cannot be efficiently represented as a set of tuples. Thus, for verifying possible solutions to SCMP, measures must be provided with a non-mass-splitting transport plan. 

\section{Further Open Questions}\label{sec:openQ}

The computational complexity of the discrete barycenter problem (DBP) remains open. Of course, the complexity of SBP -- which we studied in this paper -- only provides an upper bound on the complexity of DBP. Our proofs crucially build on the desired sparsity of the output, and thus do not reveal an immediate approach to the more general question. 

At the same time, there are some positive results about the efficient approximability of DBP. Notably, there exists a strongly polynomial $2$-approximation algorithm, and one can even guarantee a support size of at most $(\sum_{i=1}^m |P_i| -m+1)^2$, as well as the existence of a non-mass-splitting transport, for the measure giving this approximation guarantee \cite{b-17}. Other algorithms with approximation guarantees are based on an entropic regularization of the exact Wasserstein distance, which makes the underlying optimization problem strictly convex. The cost of an entropy-regularized transport plan converges to an optimal transport based on the exact Wasserstein distance in $O(e^{-w})$ \cite{bccnp-14,cdps-17,lrpc-18}, where $w$ is the regularization factor. These approaches typically require the specification of a set of possible support points for the measure to be found. For example, for grid-structured data, it is common practice to specify the original grid to obtain fast computation times; however, an exact barycenter would be supported in an $m$-times finer grid. A computation over the original grid, in fact, only leads to a convergence to a $2$-approximate solution as in \cite{b-17}. To the best of our knowledge, it is open whether entropy-regularization-based methods can be used to find an arbitrarily close approximation of an exact barycenter without specifying the possibly exponential-sized set $S$ of unique weighted means.

Second, the hardness of SCMP remains open if we do not restrict the decision to measures with a non-mass-splitting optimal transport, i.e., if one would decide {\em `Does there exist a measure $P$ with a support set $\supp(P)$ of size at most $N$ and transport cost $\phi(P)\leq \Phi$?'}. For our proof of NP-hardness of UC3P and the complexity of the related problems in this paper, it was crucial to know that mass can only be allocated to combinations in $S^*$, which is readily implied by the existence of a non-mass-splitting transport.

When considering measures that are not barycenters, the situation becomes more involved: For any measure where an optimal transport has to split mass, one can efficiently construct another measure with associated non-mass-splitting transport of strictly better cost (Alg. 2 in \cite{b-17}), but doing so results in a measure with a larger support set. Thus there is a tradeoff between the allowed size of the support set and the transport cost in SCMP. For a given, fixed combination of support size bound $N$ and transport cost bound $\Phi$, the only measure associated to a positive decision may require a mass-split in any optimal transport plan. \\

\bibliography{barycenters_literature}
\bibliographystyle{plain}

\end{document}